\documentclass[reqno]{amsart}

\usepackage{a4wide, mathrsfs, amssymb, mathtools}
\usepackage{amsthm}
\usepackage{thmtools}

\usepackage{color}
\usepackage[colorlinks=true]{hyperref}

\newcommand\citet{\cite}

\numberwithin{equation}{section}

\theoremstyle{plain}
\newtheorem{theorem}{Theorem}[section]
\newtheorem{corollary}[theorem]{Corollary}
\newtheorem{proposition}[theorem]{Proposition}
\newtheorem{lemma}[theorem]{Lemma}

\theoremstyle{definition}
\newtheorem{definition}[theorem]{Definition}
\newtheorem{example}[theorem]{Example}

\renewcommand\thmcontinues[1]{Continued}

\theoremstyle{remark}

\renewcommand\ge{\ensuremath{\geqslant}}
\renewcommand\le{\ensuremath{\leqslant}}

\newcommand\G{Gr\"{o}bner}
\newcommand\LT{\ensuremath{\mathrm{LT}}}
\newcommand\R{\ensuremath{\mathbb{R}}}
\newcommand\poly{\ensuremath{\R[x^{\vphantom{2}}_1, \dots, x^{\vphantom{2}}_r]}}
\newcommand\la{\ensuremath{\left<}}
\newcommand\ra{\ensuremath{\right>}}

\begin{document}
\title[Identifiability of Polynomial Models]{Identifiability of Polynomial Models from First Principles and via a Gr\"obner Basis Approach}

\author{J.D.~Godolphin}
\address{Department of Mathematics\\ University of Surrey\\ Guildford\\ Surrey\\ GU2 7XH\\ U.K.}
\email{j.godolphin@surrey.ac.uk}

\author{James\ D.E.~Grant}
\address{Department of Mathematics\\ University of Surrey\\ Guildford\\ Surrey\\ GU2 7XH\\ U.K.}
\email{j.grant@surrey.ac.uk}

\date{\today}

\begin{abstract}
The relationship between a set of design points and the class of hierarchical polynomial models identifiable from the design is investigated. Saturated models are of particular interest. Necessary and sufficient conditions are derived on the set of design points for specific terms to be included in leaves of the statistical fan.
A practitioner led approach to building hierarchical saturated models that are identifiable is developed.
This approach is compared to the method of model building based on Gr\"obner
bases. The main results are illustrated by examples.
\end{abstract}
\thanks{The research of the second author was partially supported by grant~\href{https://www.fwf.ac.at/forschungsradar/10.55776/PAT1996423}{PAT 1996423} of the~\href{https://www.fwf.ac.at/}{Austrian Science Fund (FWF)}}
\maketitle
\thispagestyle{empty}

\section{Introduction}\label{sect1}

In an experimental situation there are typically 
a number of controllable factors, \(r\) say, that may affect a response of interest. The relationship between the \(r\) factors and the response 
variable is often unknown and the aim is to determine a function of the factors that adequately 
models this relationship. Once established, the model can be used to determine the values or levels of each factor at which the response is optimised. Applications are widespread across fields including: manufacturing; chemical engineering; pharmaceuticals and biotechnology; agriculture. In some experiments, especially in manufacturing, the number of controllable factors can be quite large, sometimes with \(r > 20\). Key references on this topic include~\citet{wu2021experiments}, \citet{myers2016response} and~\citet{pukelsheim2006optimal}. 

In practice, the  controllable factors are transformed to \(r\) 
\emph{indeterminates}, \(x^{\vphantom{2}}_1, \ldots, x^{\vphantom{2}}_r\). For instance, the 
transformation is typically chosen so that the indeterminate values have mean zero and similar spreads. As an elementary example, consider a chemical process with yield depending on reaction time, temperature and reactant concentration. If 
two reaction times (5 minutes and 12 minutes), two temperatures (60\(^\circ\)C and 80\(^\circ\)C) and three reactant concentrations (0.25mol/dm\(^3\), 0.50mol/dm\(^3\), 0.75mol/dm\(^3\)) are 
employed, these would typically be converted to 
\(x^{\vphantom{2}}_1, x^{\vphantom{2}}_2, x^{\vphantom{2}}_3\), with \(x^{\vphantom{2}}_1\) and \(x^{\vphantom{2}}_2\) each taking 
values \(\pm 1\), and \(x^{\vphantom{2}}_3\) taking values \(0, \pm 1\). 

In this work, we consider relationships between the indeterminates and the response \(Y\) of the form:
\begin{equation}\label{intmodel}
Y = f(x^{\vphantom{2}}_1,x^{\vphantom{2}}_2,\cdots,x^{\vphantom{2}}_r) + \varepsilon, 
\end{equation}
where \(f(x^{\vphantom{2}}_1,x^{\vphantom{2}}_2,\cdots, x^{\vphantom{2}}_r)\) is a polynomial and \(\varepsilon\) is the noise or error term with zero mean and constant variance.

We address two complementary issues: identifying which polynomial models can be estimated from a given design, and constructing designs that yield estimable models with specific terms.
Preliminary concepts are introduced in \S\ref{sp}. In \S\ref{HIS}, we establish results on the existence of estimable models that are both saturated and hierarchical for a given design, and present an elementary approach to 
building such models. We develop necessary and sufficient conditions on the design points for the inclusion of selected terms in estimable saturated hierarchical models in \S\ref{s3}. 

The latter part of this work investigates an existing approach to model identification using Gr\"obner bases. The theory of Gr\"obner bases was developed by~\citet{Buchberger} for solving problems involving polynomial ideals. Over the past 25 years, a substantial body of research has emerged on the use of Gr\"obner bases in statistics, with foundational
work 
in~\citet{DiaconisSturmfels}. The application of Gr\"obner bases to experimental design began with the seminal paper by \citet{PistoneWynn}. 
For a given design, Pistone and Wynn use a Gr\"obner
basis approach to identify saturated hierarchical polynomial models for which the parameters
have unbiased estimates. Following this work, a significant body of work has expanded on the approach, exploring the use of Gr\"obner bases for model identification in various experimental design contexts. For descriptions of  work at the interface between computational commutative algebra and the design of experiments see~\citet{pistone2009methods} and~\citet{riccomagno2009short}.
Focussed applications include mixture experiments in~\citet{maruri2007description} and~\citet{giglio2001grobner}, $m$-way designs in~\citet{Wynn}, and factorial experiments in~\citet{galetto2003confounding}, \cite{evangelaras2006comparison}, \citet{aoki2012design} and~\citet{fontana2014characterization}.  

In \S\ref{sgb}, we provide a  brief overview of the theory underlying Gr\"obner bases. 
We describe the
Gr\"obner basis approach to model identifiability.
The relationship between the approaches of \S\ref{HIS} and \S\ref{sgb} is clarified, and we discuss their relative merits.

Examples are given and
developed through the paper
to illustrate the results.   

\section{Designs, Monomials and Models}\label{sp}
A~\emph{design}, \(\mathcal{D}\), is a set of \(n>1\) distinct points in \(\mathbb{R}^r\) on  indeterminates \(x^{\vphantom{2}}_1, \ldots, x^{\vphantom{2}}_r\), with
\(x^{\vphantom{2}}_j\) taking \(n^{\vphantom{2}}_j>1\) values, for $j = 1, \ldots, r$. The estimability properties of  \(\mathcal{D}\) can be obtained from the \(n \times r\) \emph{design matrix\/} \(X_{\mathcal{D}}\), with \((i,j)\)th entry, \(x^{\vphantom{2}}_{ij}\), being the value of the \(j\)th indeterminate in
the \(i\)th point. Thus:
\begin{equation}\label{desmat}
X_{\mathcal{D}}=\begin{bmatrix} x^{\vphantom{2}}_{11}&x^{\vphantom{2}}_{12}&\cdots &x^{\vphantom{2}}_{1r} \\
x^{\vphantom{2}}_{21}&x^{\vphantom{2}}_{22}&\cdots &x^{\vphantom{2}}_{2r} \\
\vdots & \vdots& \ddots & \vdots \\
x^{\vphantom{2}}_{n1}&x^{\vphantom{2}}_{n2}&\cdots &x^{\vphantom{2}}_{nr}\\
\end{bmatrix}=\begin{bmatrix}\mathbf{x}^{\vphantom{2}}_1 \, \mathbf{x}^{\vphantom{2}}_2 \cdots \mathbf{x}^{\vphantom{2}}_r \end{bmatrix}.
\end{equation}
Here, the \(i\)th row of \(X_{\mathcal{D}}\) contains the values of the indeterminates at the \(i\)th point. The \(j\)th column  contains the values of the \(j\)th
indeterminate at the \(n\) distinct points and is denoted \(\mathbf{x}^{\vphantom{2}}_j\), 
for \(j=1, \ldots, r\). 

\begin{example}[label=thm:key]
\label{example:1}
Consider design \(\mathcal{D}1\) on indeterminates \(x^{\vphantom{2}}_1,x^{\vphantom{2}}_2,x^{\vphantom{2}}_3\) with five distinct points and design matrix:
\begin{equation}\label{desmat1}
X_{\mathcal{D}1}=\begin{bmatrix} ~1&-1&~0\\
-1&-1&~1\\
~0&~1&~1\\
~0&~1&-1\\
~1&-1&-1\\
\end{bmatrix}.
\end{equation}
With \(n=5\), \(\mathcal{D}1\) has a smaller number of distinct points than would be employed in most circumstances. However, the
small size of \(X_{\mathcal{D}1}\) lends  its use for efficient demonstration of concepts which also apply to designs with larger \(n\). 
The number of observations is at least as
large as \(n\). If some runs are conducted multiple times then \(\mathcal{D}1\) could have many more than five observations but its estimability capability would still be determined by the properties of the \(5 \times 3\) matrix \(X_{\mathcal{D}1}\). 
Well known designs with multiple instances of runs include the Central Composite
Designs introduced by~\citet{BoxWilson}, and the Box--Behnken Designs of~\citet{BoxBehnken}, where multiple observations are
collected at the centre points.
\end{example}

There are two objectives: given a design, we want to determine polynomial models in the \(r\) indeterminates for which
unbiased estimates of the unknown parameters can be obtained; conversely, given knowledge of an experimental situation, we want to construct a  
design that enables us to fit relevant models. 

\medskip
We  introduce some initial concepts and terminology. 
\begin{definition} 
A~\emph{monomial\/} in the indeterminates \(x^{\vphantom{2}}_1,\ldots,x^{\vphantom{2}}_r\) is a term of the 
form \(x_1^{\alpha_1}x_2^{\alpha_2}\cdots x_r^{\alpha_r}\) with \(\alpha_j \in \mathbb{Z}_{\ge 0}\), for \(j=1, \ldots, r\). Where appropriate, we will use multi-index notation, writing such expressions as $x^{\alpha}$, where $x = \left( x^{\vphantom{2}}_1, \ldots, x^{\vphantom{2}}_r \right)$ and $\alpha = \left( \alpha^{\vphantom{2}}_1, \ldots, \alpha^{\vphantom{2}}_r \right) \in \mathbb{Z}_{\ge 0}^r$. 
\end{definition}
\noindent There is a one-to-one correspondence between monomials $x^{\alpha}$ and their corresponding ordered $r$-tuples \( \alpha = (\alpha^{\vphantom{2}}_1, \ldots, \alpha^{\vphantom{2}}_r) \in \mathbb{Z}_{\ge 0}^r\). With this correspondence in mind, we will often, for brevity, use the expression ``the monomial $\alpha$'' to mean the monomial $x^{\alpha}$. 
The total degree of
the monomial with \(r\)-tuple \(\alpha\) is \({\rm d}(\alpha)=\sum_{j=1}^r \alpha^{\vphantom{2}}_j\).  The
monomial with total degree zero has representation $1$ with corresponding \(r\)-tuple \(( 0,\ldots, 0) \). By an inductive argument using the Hockey Stick Identity, see for example~\citet{Jones}, there
are \((d+r-1)!/d!(r-1)!\) monomials of total degree \(d\). Monomials with at least two non-zero exponents are 
referred to as {\it interactions}. 

In situations involving more than one monomial, the $r$-tuple corresponding to the $k$th monomial is expressed as 
\( \alpha^{\vphantom{2}}_k = \left( \alpha^{\vphantom{2}}_{k, 1}, \ldots, \alpha^{\vphantom{2}}_{k, r} \right) \in  \mathbb{Z}_{\ge 0}^r \). 

\begin{example}[continues=thm:key]
Examples of monomials in \(x^{\vphantom{2}}_1, x^{\vphantom{2}}_2, x^{\vphantom{2}}_3\) are \(x_1^2\) and \(x^{\vphantom{2}}_1x_2^3x^{\vphantom{2}}_3\) with corresponding $3$-tuples \((2,0,0)\) and \(( 1,3,1)\), respectively. Similarly, identifying monomials with their corresponding $3$-tuples as mentioned above, by ``the monomial $(6, 7, 8)$'' we would mean the monomial $x_1^6 x_2^7 x_3^8$. 
\end{example}

All models considered are polynomial models. Hence,  \(E[Y]\) comprises a linear combination of \(p\) 
monomials in the  indeterminates and is written as:
 
\begin{equation}\label{model}
E[Y] = \sum_{k=1}^p \theta_{\alpha_k} x^{\alpha_k} 
\equiv \sum_{k=1}^p\theta_{\alpha_k} \left( \prod_{j=1}^r x_j^{\alpha_{k, j}} \right), 
\end{equation}
where \(\theta_{\alpha_k}\) is the unknown coefficient of the \(k\)th monomial $x^{\alpha_k}$ with corresponding \(r\)-tuple \( \alpha^{\vphantom{2}}_k = \left( \alpha^{\vphantom{2}}_{k, 1}, \ldots, \alpha^{\vphantom{2}}_{k, r} \right) \in  \mathbb{Z}_{\ge 0}^r \).

\begin{example}[continues=thm:key]
Polynomial models in three indeterminates include:
\begin{eqnarray*}
{\rm Model~I:~~}E[Y]&=&\theta_{( 0,0,0)} +\theta_{( 1,0,0)} x^{\vphantom{2}}_1+\theta_{( 2,0,0)} x_1^2;\\ 
{\rm Model~II:~~}E[Y]&=&\theta_{( 0,0,0)} +\theta_{( 1,0,0)}x^{\vphantom{2}}_1+\theta_{( 0,1,0)}x^{\vphantom{2}}_2+\theta_{( 0,0,1)}x^{\vphantom{2}}_3+\theta_{( 0,1,1)}x^{\vphantom{2}}_2x^{\vphantom{2}}_3+\theta_{( 0,0,2)}x_3^2;\\
{\rm Model~III:~~}E[Y]&=&\theta_{( 1,0,0)}x^{\vphantom{2}}_1+\theta_{ ( 0,0,1)}x^{\vphantom{2}}_3+\theta_{( 1,0,1)}x^{\vphantom{2}}_1x^{\vphantom{2}}_3+\theta_{( 2,0,1)}x_1^2x^{\vphantom{2}}_3. 
\end{eqnarray*}
\end{example}

\begin{definition}
Consider a design,  \(\mathcal{D}\), and a polynomial  model, \(\mathcal{M}\), in some or all of the indeterminates \(x^{\vphantom{2}}_1,\ldots,x^{\vphantom{2}}_r\). The~\emph{model matrix\/} for \(\mathcal{D}\) with
respect to \(\mathcal{M}\) is the \(n \times p\) matrix, \(W_{\mathcal{DM}}\), with \((i,k)\)th entry \( \Pi_{j=1}^r x_{ij}^{\alpha_{k, j}}\), where $x^{\vphantom{2}}_{ij}$ are defined as in~\eqref{desmat}.
\end{definition}

The \(k\)th column of \(W_{\mathcal{DM}}\) is, therefore, the $n \times 1$ column vector consisting of the values of the \(k\)th monomial $x^{\alpha_k}$ evaluated at the \(n\) points of \(\mathcal{D}\). This column vector will be referred to as the {\it monomial vector for \({\alpha_k}\)} and be denoted by \( \mathbf{x}^{\alpha_k} \) or \( \Pi_{j=1}^r \mathbf{x}_j^{\alpha_{k, j}} \), where $\mathbf{x}^{\vphantom{2}}_j$ is the $j$th column of the design matrix $X_{\mathcal{D}}$ (cf.~\eqref{desmat}). In the particular case of the monomial \(1\), the monomial vector is the \(n \times 1\) vector all of whose entries are $1$, which we denote by \(\mathbf{1}\).

\medskip
The coefficients \(\theta_{\alpha_1}, \ldots, \theta_{\alpha_p}\) are estimable from \(\mathcal{D}\) if and only if \({\rm rank}(W_{\mathcal{DM}})=p\). 

\begin{definition}\label{ide}
A polynomial model of form \eqref{model} is said to be~\emph{identifiable\/} from design \(\mathcal{D}\) if and only if \(\theta_{\alpha_1}, \ldots, \theta_{\alpha_p}\) are estimable. A model which is identifiable from \(\mathcal{D}\) and has \(p=n\) is said to be \emph{saturated}.
\end{definition}
It follows from Definition \ref{ide} that a model \(\mathcal{M}\) is identifiable from \(\mathcal{D}\) if and only if \({\rm
rank}(W_{\mathcal{DM}})=p.\) 

\begin{example}[continues=thm:key]
We consider Models I to III with respect to \(\mathcal{D}1\). Model I has \(p=3\) and
\begin{equation*}
W_{\mathcal{D}1I}=\begin{bmatrix}
\mathbf{1} & \mathbf{x}^{\vphantom{2}}_1 & \mathbf{x}_1^2\\
\end{bmatrix}=
\begin{bmatrix} 1&1&1\\
1&-1&1\\
1&0&0\\
1&0&0\\
1&1&1\\
\end{bmatrix}.
\end{equation*}
The matrix \(W_{\mathcal{D}1I}\) has rank 3, establishing that Model I is identifiable from \(\mathcal{D}1\). 

Model~II has \(p=6\) and
\begin{equation*}
W_{\mathcal{D}1II}=\begin{bmatrix}
\mathbf{1} & \mathbf{x}^{\vphantom{2}}_1 & \mathbf{x}^{\vphantom{2}}_2 & \mathbf{x}^{\vphantom{2}}_3 & \mathbf{x}^{\vphantom{2}}_2 \mathbf{x}^{\vphantom{2}}_3& \mathbf{x}_3^2 \\
\end{bmatrix}=
\begin{bmatrix}
1 & 1& -1& 0&0&0\\
1& -1& -1&1& -1&1\\
1&0&1&1&1&1\\
1&0&1&-1&-1&1\\
1&1&-1&1&-1&1\\
\end{bmatrix}.
\end{equation*}
Since rank\((W_{\mathcal{D}1II})=5\), Model II is not identifiable from \(\mathcal{D}1\). More generally, \(\mathcal{D}1\) has \(n=5\), so no model with \(p>5\) can be identifiable. 

Finally, for Model III, \(p=4\) and 
\begin{equation*}
W_{\mathcal{D}1III}=\begin{bmatrix}
\mathbf{x}^{\vphantom{2}}_1 & \mathbf{x}^{\vphantom{2}}_3 & \mathbf{x}^{\vphantom{2}}_1 \mathbf{x}^{\vphantom{2}}_3 & \mathbf{x}_1^2 \mathbf{x}^{\vphantom{2}}_3 \\
\end{bmatrix}=
\begin{bmatrix}
1 & 0&0&0\\
-1&1& -1& 1\\
0&1&0&0\\
0&-1&0&0\\
1&1&1&1\\
\end{bmatrix}.
\end{equation*}
The matrix \(W_{\mathcal{D}1III}\) has rank 4, indicating that Model III is identifiable from \(\mathcal{D}1\). 
\end{example}

From this point onwards, in seeking identifiable models we will restrict consideration to an important sub-class of models, namely hierarchical models.

\begin{definition}\label{hi}
A model is \emph{hierarchical\/} if for any monomial,\(( \alpha^{\vphantom{2}}_1, \ldots, \alpha^{\vphantom{2}}_r ) \in \mathbb{Z}_{\ge 0}^r \), in the model, every monomial \( (\beta^{\vphantom{2}}_1, \ldots, \beta^{\vphantom{2}}_r) \in \mathbb{Z}_{\ge 0}^r \) such
that \(\beta^{\vphantom{2}}_j \le \alpha^{\vphantom{2}}_j\), for \(j=1, \ldots,r\), is also contained in the model. Monomials of the form \((\beta^{\vphantom{2}}_1, \ldots, \beta^{\vphantom{2}}_r)\) are said to be \emph{constituent\/} monomials of \((\alpha^{\vphantom{2}}_1, \ldots, \alpha^{\vphantom{2}}_r)\). 
\end{definition}

\noindent Restricting attention to hierarchical models is intuitively appealing. For example, consider inclusion of the interaction \(x^{\vphantom{2}}_1x^{\vphantom{2}}_2\) in a model. The interpretation of this term  becomes challenging unless \(x^{\vphantom{2}}_1\) and \(x^{\vphantom{2}}_2\) are also included. The term \(x^{\vphantom{2}}_1x^{\vphantom{2}}_2\) is precisely an effect associated with \(x^{\vphantom{2}}_1\) and \(x^{\vphantom{2}}_2\) that cannot be
explained by taking \(x^{\vphantom{2}}_1\) and \(x^{\vphantom{2}}_2\) separately. Thus, we seek  models which are both identifiable and hierarchical. From Definition~\ref{hi}, any monomial has $1$ as a constitutent monomial and, therefore, every hierarchical model contains the monomial $1$. Further, again from Definition \ref{hi}, inclusion of the monomial
\( \alpha =( \alpha^{\vphantom{2}}_1, \ldots,\alpha^{\vphantom{2}}_r )\) in a hierarchical model, implies the inclusion of \(\Pi_{j=1}^r(\alpha^{\vphantom{2}}_j+1)\) constituent monomials of \( \alpha \). This leads to:  

\begin{theorem}\label{neci}
A necessary condition for the monomial \( \alpha =(\alpha^{\vphantom{2}}_1, \ldots, \alpha^{\vphantom{2}}_r) \) to be included in a hierarchical model which is identifiable under \(\mathcal{D}\) is that \( \Pi_{j=1}^r(\alpha^{\vphantom{2}}_j+1) \le n\).
\end{theorem}

\begin{example}[continues=thm:key]
The monomials of Model I are: $1$, \(x^{\vphantom{2}}_1\) and \(x_1^2\). Equivalently, we will say that Model I has monomial set $\{ 1, x^{\vphantom{2}}_1, x_1^2 \}$. Thus this model is hierarchical.

With monomial set \(\{1, x^{\vphantom{2}}_1, x^{\vphantom{2}}_2,x^{\vphantom{2}}_3,x^{\vphantom{2}}_2x^{\vphantom{2}}_3, x_3^2\}\), Model II is also hierarchical. As noted already, Model II is
not identifiable. However, the removal of either  \(x^{\vphantom{2}}_2x^{\vphantom{2}}_3\) or \(x_3^2\) yields a model with \(p=5\) which is hierarchical,
identifiable and saturated. 

Model III does not contain the monomial $1$ and so is not hierarchical. The model resulting from the inclusion of the constant term would also not be hierarchical since it would contain \(x_1^2x^{\vphantom{2}}_3\) but not \(x_1^2\). Further, by Theorem~\ref{neci}, the term \(x_1^2x^{\vphantom{2}}_3\) cannot be contained in any hierarchical identifiable model under \(\mathcal{D}1\).
\end{example}

\section{Hierarchical Identifiable Models}\label{HIS}

In this section we establish that, given any design, a  hierarchical identifiable saturated model exists. For brevity,
hierarchical identifiable models, and the subset of these that are saturated, will be abbreviated to {\it HI} and {\it HIS}
models, respectively. The following result will be used to establish the existence of at least one {\it HIS} model for every design. As well as establishing existence, the result provides insight into the nature of models that are identifiable and saturated.

\begin{proposition}\label{van}
For any design, \(\mathcal{D}\), a set of monomials \( \alpha^{\vphantom{2}}_1, \ldots, \alpha^{\vphantom{2}}_{n-1} \) exists with 
\({\rm
d}(\alpha^{\vphantom{2}}_j)=j\), such that the monomial vectors \(\boldsymbol{1}, \mathbf{x}^{\alpha_1}, \ldots, \mathbf{x}^{\alpha_{n-1}} \)
form a basis for \(\mathbb{R}^n\).
\end{proposition}

\begin{proof}
Since the rows of \(X_{\mathcal{D}}\) are distinct, some linear combination of the columns of
\(X_{\mathcal{D}}\) can be found that yields a vector \(\mathbf z=(z_1, \ldots, z_n)^T\) with \(n\) distinct elements. Thus 
\(\mathbf z = \sum_{j=1}^r c_j \mathbf{x}^{\vphantom{2}}_j \), for some real coefficients \(c_1, \ldots, c_r\), not all zero.
Since the entries of \(\mathbf z\) are all different, the \(n \times n\) matrix:
\begin{equation*}
Z=\begin{bmatrix}
\mathbf{1} & \mathbf{z} & \mathbf{z}^2 &\cdots & \mathbf{z}^{n-1}\\
\end{bmatrix}=
\begin{bmatrix} 1&z_1&z_1^2&\cdots &z_1^{n-1} \\
1&z_2&z_2^2&\cdots &z_2^{n-1} \\
\vdots & \vdots & \vdots& \ddots & \vdots \\
1&z_n&z_n^2&\cdots &z_n^{n-1} \\
\end{bmatrix}
\end{equation*}
is a Vandermonde matrix and has rank \(n\) (see \S 13.6 of~\citet{Harville}). Hence a basis for \(\mathbb{R}^n\) is given by
\(\{\mathbf{1}, \mathbf{z},\mathbf{z}^2,\ldots, \mathbf{z}^{n-1} \}\). Here \( \mathbf{z}^q \) is a linear combination of
monomial vectors in \(x^{\vphantom{2}}_1, \ldots, x^{\vphantom{2}}_r\) with total degree \(q\). A monomial vector corresponding to a monomial of total degree
one, \(\mathbf{x}^{\alpha_1}\) say, can be found to replace \(\mathbf{z}\) in the basis, so that \(\{\mathbf{1}, 
\mathbf{x}^{\alpha_1}, \mathbf{z}^2, \ldots, \mathbf{z}^{n-1} \}\) is an updated basis for \(\mathbb{R}^n\). Similarly, in this
updated basis, a vector relating to a monomial of total degree two, \(\mathbf{x}^{\alpha_2} \) say, can be found to replace
\( \mathbf{z}^2 \),  to further update the basis to \( \{\mathbf{1}, \mathbf{x}^{\alpha_1}, \mathbf{x}^{\alpha_2}, \mathbf{z}^3, \ldots,
\mathbf{z}^{n-1} \}\). Repeating this process, we progressively replace \(\mathbf{z}^3, \ldots, \mathbf{z}^{n-1} \) by \( \mathbf{x}^{\alpha_3}, \ldots, \mathbf{x}^{\alpha_{n-1}} \) giving a basis \(\{\mathbf{1}, \mathbf{x}^{\alpha_1}, \ldots, \mathbf{x}^{\alpha_{n-1}} \}\) for \(\mathbb{R}^n\), as required. 
\end{proof}

\noindent Proposition~\ref{van} leads immediately to a result on monomials with total degree exceeding \(n-1\):
\begin{corollary}\label{cor1}
Consider a design \(\mathcal{D}\). Let \(\alpha\) be a monomial with \({\rm d}(\alpha) \ge n\). The monomial vector for \(\alpha\) in the design \(\mathcal{D}\) can be expressed as a
linear combination of vectors of monomials each with total degree strictly less than \(n\).  
\end{corollary}

\begin{example}[continues=thm:key]
For \(\mathcal{D}1\), any monomial \(\alpha\) with \({\rm d}(\alpha) \ge 5\), such
as \(x_1^2x_2^2x^{\vphantom{2}}_3\), will have monomial vector which is the linear combination of vectors corresponding to monomials with total
degree less than five.
\end{example}

\medskip
\noindent{}One of the main results of our paper is the following: 

\begin{theorem}\label{hier}
Let \(\mathcal{D}\) be a design. Then, there exists a {\it HIS} model that is supported by \(\mathcal{D}\).
\end{theorem}
\begin{proof}
The proof is by construction of a basis for $\mathbb{R}^n$. The method is as follows: 

\medskip\noindent {\bf Step 0}: Let \(S=\{\mathbf{1}\}\).
 
\smallskip\noindent {\bf Step 1}: Let \(T_1 =\{ \mathbf{x}^{\vphantom{2}}_1,\ldots, \mathbf{x}^{\vphantom{2}}_r \}\), the set of vectors (in $\mathbb{R}^n$) of the monomials of total degree
one. Consider the members of \(T_1\), in any order, one at a time. If a vector is linearly independent of the vectors already in
\(S\), then it is added to \(S\). Otherwise it is discarded.

\smallskip\noindent{\bf Step 2}: Let \(T_2\) comprise the set of vectors corresponding to monomials of total degree 2 for which the
vectors for any constituent monomials are in \(S\). Consider the members of \(T_2\) in any order, one at a
time. If a vector is linearly independent of those already in \(S\), it is added to \(S\). Otherwise it is discarded. 

\smallskip\noindent 
The method progresses step by step, the \(q\)th step being:  

\smallskip\noindent{\bf Step $\boldsymbol{q}$}:  Let \(T_{q}\) comprise the set of vectors corresponding to monomials of total degree \(q\) for
which the vectors for any constituent monomials  are in \(S\). Consider the members 
of \(T_{q}\), in any order, one at a time. If a vector is linearly independent of those vectors already in \(S\), then it is
added to \(S\). Otherwise it is discarded. 

\smallskip\noindent By the end of Step \(q\) each monomial of total degree \(q\) either has its vector included in \(S\),
or its monomial vector is established as being a linear combination of vectors in \(S\). Thus, at the end of Step \(q\),
any monomial of total degree less than or equal to \(q\) for which the corresponding vector is not contained in \(S\) will have monomial
vector which is a linear combination of vectors in \(S\). By construction, at each stage the vectors in \(S\) correspond to a hierarchical set of
monomials.

\smallskip\noindent This process terminates once \(S\) contains \(n\) monomial vectors. By Proposition~\ref{van}, this occurs by the end of Step \(n-1\), at the latest. Moreover, these $n$ monomial vectors form a basis for $\mathbb{R}^n$ and yield an $n \times n$ model matrix that is of full rank. Thus, the model is a {\it HIS} model in $\mathcal{D}$ and
consequently the set of {\it HIS} models supported by \(\mathcal{D}\) is non-empty, as required. 
\end{proof}

\noindent Theorem \ref{hier} establishes that for every design, \(\mathcal{D}\), there is at least one  {\it HIS} model. For a given design, Theorem \ref{hier} provides a straightforward procedure for identifying a {\it HIS} model. The model is constructed monomial 
by monomial, with monomials of degree \(q\) included before those of degree \(q+1\). At each monomial degree, input from those with expert knowledge of the process being 
modelled can be used to order the monomials according to their perceived relevance. Using the
construction method in the proof of Theorem~\ref{hier}, it may be possible to generate multiple {\it HIS} models for a given design.  
The set of {\it HIS} models supported by \(\mathcal{D}\) is referred to as the
{\it statistical fan} of \(\mathcal{D}\), with individual models in the fan called {\it leaves}.
As will be demonstrated by consideration of \(\mathcal{D}1\), the approach of Theorem~\ref{hier} does not always obtain every {\it HIS} model. 

\begin{example}[continues=thm:key]
Using the approach of Theorem \ref{hier}, a number of {\it HIS} models that are supported by \(\mathcal{D}1\) can be found by constructing bases for \(\mathbb{R}^5\)
comprising monomial vectors. Each basis is formed vector by vector, starting with \(S=\{\mathbf{1}\}\). For Step 1, set
\(T_1=\{ \mathbf{x}^{\vphantom{2}}_1,\mathbf{x}^{\vphantom{2}}_2,\mathbf{x}^{\vphantom{2}}_3 \}\). Since the \(5 \times 4\) matrix \((\mathbf{1} \,\, \mathbf{x}^{\vphantom{2}}_1 \,\, \mathbf{x}^{\vphantom{2}}_2 \,\,
\mathbf{x}^{\vphantom{2}}_3 )\) has rank $4$, at the end of Step 1 we have \(S=\{\mathbf{1}, \, \mathbf{x}^{\vphantom{2}}_1, \, \mathbf{x}^{\vphantom{2}}_2, \,
\mathbf{x}^{\vphantom{2}}_3 \}\), whichever order the terms of \(T_1\) are taken in. For Step 2, take \( T_2 = \{ \mathbf{x}_1^2, \mathbf{x}_2^2,
\mathbf{x}_3^2, \mathbf{x}^{\vphantom{2}}_1 \mathbf{x}^{\vphantom{2}}_2, \mathbf{x}^{\vphantom{2}}_1 \mathbf{x}^{\vphantom{2}}_3, \mathbf{x}^{\vphantom{2}}_2 \mathbf{x}^{\vphantom{2}}_3 \}\). Of the vectors in \(T_2,\) \( \mathbf{x}_3^2, \mathbf{x}^{\vphantom{2}}_1 \mathbf{x}^{\vphantom{2}}_2,
\mathbf{x}^{\vphantom{2}}_1 \mathbf{x}^{\vphantom{2}}_3, \mathbf{x}^{\vphantom{2}}_2 \mathbf{x}^{\vphantom{2}}_3 \) are each linearly independent of those in \(S\) and any one can be added to 
\(S\), giving rise to four {\it HIS} models. For example, one such model contains the terms: $1, x^{\vphantom{2}}_1, x^{\vphantom{2}}_2, x^{\vphantom{2}}_3, x_3^2$. However, the {\it HIS} model with terms: \(1, x^{\vphantom{2}}_1, x_1^2, x^{\vphantom{2}}_3,
x^{\vphantom{2}}_1 x^{\vphantom{2}}_3,\) is supported by \(\mathcal{D}1\), but cannot be  obtained directly via the method of Theorem~\ref{hier}. We return to this case in the next Section. 
\end{example}

\section{Conditions on Monomials contained in {\it HI} and {\it HIS} Models}\label{s3}

Theorem \ref{neci} gives a necessary condition on a monomial for its inclusion in
a {\it HI} model under a given design \(\mathcal{D}\). In this section, we first consider subsets of indeterminates, finding further necessary and/or sufficient 
conditions on monomials 
that can be contained in {\it HI} and {\it HIS} models
under \(\mathcal{D}\). We then establish a sufficient condition for a monomial to be included in every leaf of the statistical fan of a design $\mathcal{D}$. 

Let \(U\) be a subset of \(l<r\) indeterminates. Without loss of generality, let these be \(x^{\vphantom{2}}_1, \ldots, x^{\vphantom{2}}_l\). Form \(X_{\mathcal{D}U}\), a \(n^{\vphantom{2}}_U \times l\) matrix, from \(X_{\mathcal{D}}\), by deleting columns
corresponding to indeterminates not in \(U\) and, if necessary, deleting any repeats of rows, so
that all rows are distinct. Let \(\mathcal{D}_U\) denote the design with design matrix \(X_{\mathcal{D}U}\). 
The {\it HI} models for \(\mathcal{D}_{U}\) will be precisely the {\it HI} models for \(\mathcal{D}\) in \(x^{\vphantom{2}}_1, \ldots, x^{\vphantom{2}}_l\)
only, that is, {\it HI} models in which all monomials have zero exponents for \(x^{\vphantom{2}}_{l+1}, \ldots, x^{\vphantom{2}}_r\).
The following result is obtained using an approach which parallels that leading to Theorem \ref{hier} and is given without
proof.
\begin{lemma}\label{th:idsub}
For the design \(\mathcal{D}\) and indeterminate set \(U=\{x^{\vphantom{2}}_1, \ldots, x^{\vphantom{2}}_l\}\), where \(l <r\), let the order of \(X_{\mathcal{D}U}\) be \(n^{\vphantom{2}}_U \times l\). Then \(\mathcal{D}\) supports at least one {\it HI} model involving \(n^{\vphantom{2}}_U\) monomials in
indeterminates in \(U\) only, but supports no {\it HI} models involving more than \(n^{\vphantom{2}}_U\) monomials in
indeterminates in \(U\) only.
\end{lemma}
Lemma \ref{th:idsub} gives an achievable upper bound of \(n^{\vphantom{2}}_U\) for the number of monomials (including 1) contained in a {\it HI} model, supported by \(\mathcal{D}\), for which each monomial has \(\alpha^{\vphantom{2}}_{l+1}=\cdots=\alpha^{\vphantom{2}}_r=0\). The result can clearly be generalised to cover all subsets of \(l<r\) indeterminates. We demonstrate the result by consideration of \(\mathcal{D}1\).

\begin{example}[continues=thm:key]
Consider the set \(U1=\{x^{\vphantom{2}}_1,x^{\vphantom{2}}_2\}\). From the first two columns of  \(X_{\mathcal{D}1}\) given in \eqref{desmat1}, the design matrix \(X_{\mathcal{D}1U1}\) is: 

\begin{equation}\label{desmat1u}
X_{\mathcal{D}1U1}=\begin{bmatrix} ~1&-1\\
-1&-1\\
~0&~1\\
\end{bmatrix},
\end{equation}
giving \(n^{\vphantom{2}}_{U1}=3\). By Lemma \ref{th:idsub}, there is at least one {\it HI} model in three monomials involving only \(x^{\vphantom{2}}_1\) and \(x^{\vphantom{2}}_2\) which is supported by
\(\mathcal{D}1\). From earlier investigation of Models I and II, it is
known, or deduced, that models with monomial sets \(\{1, x^{\vphantom{2}}_1, x_1^2\}\) and \(\{1, x^{\vphantom{2}}_1, x^{\vphantom{2}}_2\}\) are {\it HI} under \(\mathcal{D}1\). However, also by Lemma~\ref{th:idsub}, models with monomial sets with four or more monomials in \(x^{\vphantom{2}}_1\)
and \(x^{\vphantom{2}}_2\), such as \(\{1, x^{\vphantom{2}}_1, x^{\vphantom{2}}_2, x^{\vphantom{2}}_1x^{\vphantom{2}}_2\}\) and \(\{1, x^{\vphantom{2}}_1, x^{\vphantom{2}}_2, x_2^2\}\), are~\emph{not\/} identifiable under \(\mathcal{D}1\). This is confirmed by calculating the
ranks of the corresponding \(5 \times 4\) model matrices. 
\end{example}

Lemma \ref{th:idsub} has a number of important applications. For a subset comprising a single indeterminate, that is \(U=\{x^{\vphantom{2}}_j\}\), there is only one
possible {\it HI} model with \(n^{\vphantom{2}}_U=n^{\vphantom{2}}_j\) terms, namely that with terms \(1, x^{\vphantom{2}}_j, \ldots, x_j^{n_j-1}\). Hence:

\begin{corollary}\label{idonecor}
For any indeterminate, \(x^{\vphantom{2}}_j\), design \(\mathcal{D}\) supports the {\it HI} model with monomial set \( \left\{ 1, x^{\vphantom{2}}_j, \ldots, x_j^{n_j-1} \right\}\).
\end{corollary}

If a proper subset of $l < r$ indeterminates \(U\) with \(n^{\vphantom{2}}_U=n\) exists, then, by Lemma~\ref{th:idsub}, \( \mathcal{D} \) supports a {\it HI} model involving \(n\) monomials in \(U\) only. Such a model is a {\it HIS} model by definition and signifies the existence of a {\it HIS} model in the $l$ indeterminates in $U$:

\begin{corollary}\label{newc}
If there exists a proper subset of indeterminates \(U\) with \(n^{\vphantom{2}}_U=n\), then $\mathcal{D}$ supports a  {\it HIS} model involving monomials with indeterminates only from \(U\).
\end{corollary}

We demonstrate the result of Corollary \ref{newc} by reference to the design \(\mathcal{D}1\) of Example \ref{example:1}.

\begin{example}[continues=thm:key]
For \(\mathcal{D}1\), consider \(U2=\{x^{\vphantom{2}}_1, x^{\vphantom{2}}_3\}\). This set has \(n^{\vphantom{2}}_{U2}=5=n\). By Corollary \ref{newc} it follows that a {\it HIS} model exists in \(x^{\vphantom{2}}_1, x^{\vphantom{2}}_3\). Using the approach of Theorem \ref{hier} on the design matrix \(X_{\mathcal{D}1U2}\) yields the {\it HIS} model noted earlier, i.e. the model with terms: \(1, x^{\vphantom{2}}_1, x_1^2, x^{\vphantom{2}}_3,
x^{\vphantom{2}}_1 x^{\vphantom{2}}_3\).
\end{example}

In general, even a small design can result in a statistical fan with a large number of leaves, i.e. {\it HIS} models. Within the set 
of {\it HIS} models there can be a wide range in the indeterminates and monomials included. This allows scope for the 
practitioner to select models that seem appropriate, given knowledge of the experimental situation, for further 
investigation. In practice, we would expect a practitioner to consider the inclusion of monomials in increasing degree. For 
monomials of the same degree, inclusion would be considered in order of perceived importance. Thus, for a given design there would be no need to obtain all {\it HIS} models. 
However, it is illuminating to demonstrate the number and scope of {\it HIS} models generated by Theorem \ref{hier} and Corollary \ref{newc} from a design with \(n=9\). 

\begin{example}[label=thm:key2]
\label{example:4}
\citet{JonesNachtsheim} propose a class of screening designs for indeterminates having \(n^{\vphantom{2}}_i=3\), for \(i=1,\ldots,r\), and with \(n=2r+1\). We consider a design, \(\mathcal{D}2\), with \(r=4\) from this class, i.e.
on indeterminates \(x^{\vphantom{2}}_1,x^{\vphantom{2}}_2,x^{\vphantom{2}}_3,x^{\vphantom{2}}_4\). The design matrix is:
\begin{equation}\label{desmat2}
X_{\mathcal{D}2}=\begin{bmatrix} ~0&~1&-1&-1\\
-1&~0&-1&~1\\
-1&-1&~0&-1\\
-1&~1&~1&~0\\
~0&~0&~0&~0\\
~0&-1&~1&~1\\
~1&~0&~1&-1\\
~1&~1&~0&~1\\
~1&-1&-1&~0\\
\end{bmatrix}.
\end{equation}
The approach of Theorem \ref{hier} is used to investigate the {\it HIS} models supported by \(\mathcal{D}2\). The \(9 \times 5\) matrix \( [ \mathbf{1} \,\, \mathbf{x}^{\vphantom{2}}_1 \,\, \mathbf{x}^{\vphantom{2}}_2 \,\,
\mathbf{x}^{\vphantom{2}}_3\,\, \mathbf{x}^{\vphantom{2}}_4 ] \) has rank $5$. Thus, at the end of Step 1 we have \(S=\{\mathbf{1}, \, \mathbf{x}^{\vphantom{2}}_1, \, \mathbf{x}^{\vphantom{2}}_2, \,
\mathbf{x}^{\vphantom{2}}_3, \, \mathbf{x}^{\vphantom{2}}_4\}\). For Step 2, take \( T_2 = \{ \mathbf{x}_1^2, \mathbf{x}_2^2,
\mathbf{x}_3^2, \mathbf{x}_4^2, \mathbf{x}^{\vphantom{2}}_1 \mathbf{x}^{\vphantom{2}}_2, \mathbf{x}^{\vphantom{2}}_1 \mathbf{x}^{\vphantom{2}}_3, \mathbf{x}^{\vphantom{2}}_1 \mathbf{x}^{\vphantom{2}}_4, \mathbf{x}^{\vphantom{2}}_2 \mathbf{x}^{\vphantom{2}}_3, \mathbf{x}^{\vphantom{2}}_2 \mathbf{x}^{\vphantom{2}}_4, \mathbf{x}^{\vphantom{2}}_3 \mathbf{x}^{\vphantom{2}}_4 \}\). A search through the \(10!/6!4!\) sets of four members of \(T_2\) yields 125 {\it HIS} models. 
Further, any subset, \(U3\) say, of three indeterminates has \(n^{\vphantom{2}}_{U3}=9\). By Corollary \ref{newc}, every such subset yields at least one {\it HIS} model. Investigation of the subsets of three indeterminates, again via the approach of Theorem \ref{hier}, yields 30 {\it HIS} models in each case, giving 120 {\it HIS} models in three indeterminates. 
Similarly, any subset, \(U2\) say, of two indeterminates has \(n^{\vphantom{2}}_{U2}=9\), containing all 9 combinations of the two indeterminates and gives rise to a single {\it HIS} model. Thus, the sets of two indeterminates  yield, in total, $6$ {\it HIS} models. Finally, any subset, $U1$ say, containing a single indeterminate has $n_{U1} = 3$, and does not give rise to any {\it HIS} model supported by $\mathcal{D}2$. Therefore, in total, our methods 
give $251$ distinct {\it HIS} models that are supported by  \(\mathcal{D}2\). We return to this example in \S\ref{sgb}.
\end{example}

Application of Lemma \ref{th:idsub} to all subsets of indeterminates gives the following result which generalises Theorem \ref{neci}.
\begin{theorem}\label{subset}
A necessary condition for the monomial \(\alpha=(\alpha^{\vphantom{2}}_{1}, \ldots, \alpha^{\vphantom{2}}_{r})\), to be included in a {\it HI} model is \( \Pi_{i\in U}(\alpha^{\vphantom{2}}_i+1) \le n^{\vphantom{2}}_U\) for every subset of indeterminates \(U\).
\end{theorem}

\begin{proof} 
Consider any subset \(U\) of \(l\) indeterminates. If the monomial \(\alpha=(\alpha^{\vphantom{2}}_{1},
\ldots, \alpha^{\vphantom{2}}_{r})\) is contained in a {\it HI} model supported by \(\mathcal{D}\) then the monomial  
\(\beta=(\beta^{\vphantom{2}}_{1}, \ldots, \beta^{\vphantom{2}}_{r})\) must also be contained in the model, where \(\beta^{\vphantom{2}}_i=\alpha^{\vphantom{2}}_i\) if \(x^{\vphantom{2}}_i \in U\)
and otherwise \(\beta^{\vphantom{2}}_i=0\). Similarly, all monomials which are constituent to \(\beta\) will be in the
model.  There are \( \Pi_{i\in U}(\alpha^{\vphantom{2}}_i+1)\) such monomials. If these monomials are contained in a {\it HI} model supported
by \(\mathcal{D}\), then the model comprising just these \( \Pi_{i\in U}(\alpha^{\vphantom{2}}_i+1)\)  monomials will also be {\it HI} for \(\mathcal{D}\). Hence, by Lemma
\ref{th:idsub}, \( \Pi_{i\in U}(\alpha^{\vphantom{2}}_i+1) \le n^{\vphantom{2}}_U.\) This holds for every \(U\) and the result follows.
\end{proof}

Theorem \ref{subset} applies to any subset of \(l\) indeterminates. 
Applied to the set \(U=\{x^{\vphantom{2}}_j\}\) gives the following result, which complements Corollary~\ref{idonecor}:

\begin{corollary}\label{xyz}
A necessary condition for the monomial \(\alpha=(\alpha^{\vphantom{2}}_{1}, \ldots, \alpha^{\vphantom{2}}_{r})\), to be included in a {\it HI} model is
\(\alpha^{\vphantom{2}}_{j}\le n^{\vphantom{2}}_j-1\) for \(j=1,\ldots, r\).
\end{corollary}
For a given design, Theorem \ref{subset} gives information on subsets of monomials that can occur together in {\it HI} models. It also provides guidance on design construction.

\begin{example}[continues=thm:key]
Consider the set \(U=\{x^{\vphantom{2}}_2,x^{\vphantom{2}}_3\}\) which has \(n^{\vphantom{2}}_U=5\). By Lemma \ref{th:idsub}, \(\mathcal{D}\) supports a \(HI\) model in 5 monomials all with \(\alpha^{\vphantom{2}}_1=0\). Further, since
\(n^{\vphantom{2}}_2=2\) and \(n^{\vphantom{2}}_3=3\), by Corollary \ref{xyz}, \(\alpha^{\vphantom{2}}_2 \le 1\) and \(\alpha^{\vphantom{2}}_3 \le 2\) for all monomials \(\alpha=(\alpha^{\vphantom{2}}_1,\alpha^{\vphantom{2}}_2,\alpha^{\vphantom{2}}_3)\) in a {\it HI} model. There is exactly one set of five
monomials that satisfy all three conditions on the \(\alpha^{\vphantom{2}}_i\), namely \(\{1, x^{\vphantom{2}}_2, x^{\vphantom{2}}_3, x^{\vphantom{2}}_2 x^{\vphantom{2}}_3,x_3^2\}\) and therefore  \(\mathcal{D}\) supports a {\it HIS} model in these terms.

\end{example}
This deduction from Example~\ref{example:1} would apply to any design in which a set \(U=\{x^{\vphantom{2}}_i,x^{\vphantom{2}}_j\}\) had \(n^{\vphantom{2}}_i=2\), \(n^{\vphantom{2}}_j=3\) and \(n^{\vphantom{2}}_U=5\). This points towards how Theorem \ref{subset} can be used to aid in design planning, as demonstrated with the following example. 

\begin{example}
\label{example:2}
An experiment is planned with two indeterminates: \(x^{\vphantom{2}}_1\) and \(x^{\vphantom{2}}_2\), with \(n^{\vphantom{2}}_1=4\) and \(n^{\vphantom{2}}_2=3\). A design with \(n=12\) would include a run for each combination of \(x^{\vphantom{2}}_1\) and \(x^{\vphantom{2}}_2\). By Theorem \ref{hier} and Corollary \ref{xyz}, such a design would yield the only {\it HIS}
model in 12 terms, i.e. the model containing all monomials \(\alpha=(\alpha^{\vphantom{2}}_1,\alpha^{\vphantom{2}}_2)\) with \(\alpha^{\vphantom{2}}_1 \le 3\) and \(\alpha^{\vphantom{2}}_2 \le 2\). 

It is interesting and illuminating to consider the case where resources do not accommodate as many as 12 runs, or where the factors are {\it difficult to change} so \(n<12\) but some runs may be repeated. A multiplication table between the monomials in \(x^{\vphantom{2}}_1\) and those in \(x^{\vphantom{2}}_2\) follows:
\begin{equation}\label{mtab}
\begin{tabular}
{c|ccc}
\(\times\) & 1 & \(x^{\vphantom{2}}_2\) & \(x_2^2\)\\[1mm] \hline \\[-3.3mm]
1 & 1 & \(x^{\vphantom{2}}_2\) & \(x_2^2\)\\[0.6mm]
\(x^{\vphantom{2}}_1\) & \(x^{\vphantom{2}}_1\) & \(x^{\vphantom{2}}_1x^{\vphantom{2}}_2\) & \(x^{\vphantom{2}}_1x_2^2\)\\[0.8mm]  
\(x_1^2\) & \(x_1^2\) & \(x_1^2x^{\vphantom{2}}_2\) & \(x_1^2x_2^2\)\\[0.8mm]
\(x_1^3\) & \(x_1^3\) & \(x_1^3x^{\vphantom{2}}_2\) & \(x_1^3x_2^2\)\\[0.5mm]
\end{tabular}
\end{equation}
With this arrangement, the constituent monomials for any selected
monomial are exactly the monomial itself together with any which are to the left and/or above. For example, \(x_1^2x^{\vphantom{2}}_2\) has 6 constituent monomials: \(x_1^2x^{\vphantom{2}}_2, 1,x^{\vphantom{2}}_2, x^{\vphantom{2}}_1, x^{\vphantom{2}}_1x^{\vphantom{2}}_2, x_1^2\). This is consistent with the number of constituent monomials for \(\alpha=(\alpha^{\vphantom{2}}_1,\alpha^{\vphantom{2}}_2)\) being \((\alpha^{\vphantom{2}}_1+1)(\alpha^{\vphantom{2}}_2+1)\). 

Consider a design with 8 runs. From Theorem \ref{hier} a {\it HIS} model exists. Inspection of the multiplication table in \eqref{mtab} reveals that there are 4 possible hierarchical models of 8 terms. The four monomial configurations are displayed in Table~\ref{tab8run}. 
For example, the top-left configuration corresponds to a {\it HIS} model with monomial set $\left\{ 1, x^{\vphantom{2}}_2, x_2^2, x^{\vphantom{2}}_1, x^{\vphantom{2}}_1 x^{\vphantom{2}}_2, x^{\vphantom{2}}_1 x_2^2, x_1^2, x_1^2 x^{\vphantom{2}}_2 \right\}$. 
For a given design with \(n=8\), the leaves of the statistical fan will include at least one of
these models. It is notable that the five monomials: \(1, x^{\vphantom{2}}_1, x^{\vphantom{2}}_2, x_1^2, x^{\vphantom{2}}_1x^{\vphantom{2}}_2\) are guaranteed to occur in every {\it HIS} model with \(n=8\). There are \(12!/8!4!=495\)
possible designs, which is a small enough number to enable evaluation of each design. Of the 495 designs, 36 have a statistical fan which accommodates all four models.

\begin{table}[h]
\begin{center}
\caption{The set of possible leaves for Example~\ref{example:2} with \(n=8\)}
\label{tab8run} 
\begin{tabular}
{c|ccc}
\(\times\) & 1 & \(x^{\vphantom{2}}_2\) & \(x_2^2\)\\[1mm] \hline  \\[-3.3mm]
1 & 1 & \(x^{\vphantom{2}}_2\) & \(x_2^2\)\\[0.6mm]
\(x^{\vphantom{2}}_1\) & \(x^{\vphantom{2}}_1\) & \(x^{\vphantom{2}}_1x^{\vphantom{2}}_2\) & \(x^{\vphantom{2}}_1x_2^2\)\\[0.8mm]
\(x_1^2\) & \(x_1^2\) & \(x_1^2x^{\vphantom{2}}_2\) & -\\[0.8mm]
\(x_1^3\) & - & - & -\\
\end{tabular}~~~~~~~
\begin{tabular}
{c|ccc}
\(\times\) & 1 & \(x^{\vphantom{2}}_2\) & \(x_2^2\)\\[1mm] \hline  \\[-3.3mm]
1 & 1 & \(x^{\vphantom{2}}_2\) & \(x_2^2\)\\[0.6mm]
\(x^{\vphantom{2}}_1\) & \(x^{\vphantom{2}}_1\) & \(x^{\vphantom{2}}_1x^{\vphantom{2}}_2\) & \(x^{\vphantom{2}}_1x_2^2\)\\[0.8mm]  
\(x_1^2\) & \(x_1^2\) & - & -\\[0.8mm]
\(x_1^3\) & \(x_1^3\) & - & -\\
\end{tabular}

\vspace{0.4in}

\begin{tabular}
{c|ccc}
\(\times\) & 1 & \(x^{\vphantom{2}}_2\) & \(x_2^2\)\\[1mm] \hline  \\[-3.3mm]
1 & \(1\) & \(x^{\vphantom{2}}_2\) & \(x_2^2\)\\[0.6mm]
\(x^{\vphantom{2}}_1\) & \(x^{\vphantom{2}}_1\) & \(x^{\vphantom{2}}_1x^{\vphantom{2}}_2\) & -\\[0.8mm]
\(x_1^2\) & \(x_1^2\) & \(x_1^2x^{\vphantom{2}}_2\) & -\\[0.8mm]
\(x_1^3\) & \(x_1^3\) & - & -\\
\end{tabular}~~~~~~~
\begin{tabular}
{c|ccc}
\(\times\) & 1 & \(x^{\vphantom{2}}_2\) & \(x_2^2\)\\[1mm] \hline  \\[-3.3mm]
1 & 1 & \(x^{\vphantom{2}}_2\) & -\\[0.6mm]
\(x^{\vphantom{2}}_1\) & \(x^{\vphantom{2}}_1\) & \(x^{\vphantom{2}}_1x^{\vphantom{2}}_2\) & -\\[0.8mm]
\(x_1^2\) & \(x_1^2\) & \(x_1^2x^{\vphantom{2}}_2\) & -\\[0.8mm]
\(x_1^3\) & \(x_1^3\) & \(x_1^3x^{\vphantom{2}}_2\) & -\\
\end{tabular}
\end{center}
\end{table}

\end{example}

The observation that, for the parameters of Example~\ref{example:2}, a subset of five monomials occurs in every {\it HIS} model with \(n=8\) may be formalised into the following general result.  

\begin{theorem}\label{inc}
A sufficient condition for the monomial \( \alpha = (\alpha^{\vphantom{2}}_{1}, \ldots, \alpha^{\vphantom{2}}_r)\), with \(\alpha^{\vphantom{2}}_i<n^{\vphantom{2}}_i\) for \(i=1, \ldots, r\), to be included in every leaf of the statistical fan of design \(\mathcal{D}\) is \(n \ge G(\alpha)\), where: 
\begin{equation}\label{Geq}
G(\alpha) = \Pi_{i=1}^r n^{\vphantom{2}}_i- \Pi_{i=1}^r (n^{\vphantom{2}}_i-\alpha^{\vphantom{2}}_i)+1.
\end{equation}
\end{theorem}

\begin{proof}
For a given monomial \(\alpha\), let \(V=\{(\beta^{\vphantom{2}}_1,\ldots,\beta^{\vphantom{2}}_r):\alpha^{\vphantom{2}}_i \le \beta^{\vphantom{2}}_i < n^{\vphantom{2}}_i\}\). The set \(V\) has
cardinality \(\Pi_{i=1}^r(n^{\vphantom{2}}_i-\alpha^{\vphantom{2}}_i)\) and comprises exactly those monomials that can occur in a {\it HI} model and which have \(\alpha\) as
constituent monomial. Thus, any leaf of the statistical fan of \(\mathcal{D}\) which contains a member of \(V\) must necessarily contain 
\(\alpha\). By Corollary \ref{xyz}, the number of monomials available for inclusion in a model is \(\Pi_{i=1}^r n^{\vphantom{2}}_i\). If \(n\) exceeds 
\(\Pi_{i=1}^r n^{\vphantom{2}}_i- \Pi_{i=1}^r(n^{\vphantom{2}}_i-\alpha^{\vphantom{2}}_i)\) it follows that at least one member of \(V\) is contained in every leaf of the statistical fan of
\(\mathcal{D}\) and hence that \(\alpha\) itself is in every {\it HIS} model supported by \(\mathcal{D}\). The result follows.
\end{proof}

Given \(n^{\vphantom{2}}_1, \ldots,n^{\vphantom{2}}_r\), Theorem \ref{inc} provides the lowest value of \(n\) that guarantees existence of any selected monomial in
every leaf of the statistical fan. This is a useful tool which aids in design planning when specific terms are thought to be fundamental.  The approach is
demonstrated:

\begin{example}
\label{example:3}
An experiment is planned on four factors represented by indeterminates \(x^{\vphantom{2}}_1,x^{\vphantom{2}}_2,x^{\vphantom{2}}_3, x^{\vphantom{2}}_4\). The first two factors are thought to
have a linear effect on the response whilst the third and fourth factors may have a quadratic effect. Hence we use \(n^{\vphantom{2}}_1=n^{\vphantom{2}}_2=2\) and \(n^{\vphantom{2}}_3=n^{\vphantom{2}}_4=3\). Table~\ref{tab1} lists the monomials that satisfy Corollary \ref{xyz} and so could be present in a leaf of the statistical fan of a design, in increasing order of
\(G(\alpha)\). Thus, any design with \(n\ge 19\) would produce a statistical fan in which each leaf contained all monomials of degree one. Likewise, every leaf in a statistical fan of any design with \(n\ge 25\) would contain the terms: \(
    1,x^{\vphantom{2}}_1,x^{\vphantom{2}}_2,x^{\vphantom{2}}_3,x^{\vphantom{2}}_4, x^{\vphantom{2}}_3x^{\vphantom{2}}_4,x^{\vphantom{2}}_1x^{\vphantom{2}}_3,x^{\vphantom{2}}_1x^{\vphantom{2}}_4,x^{\vphantom{2}}_2x^{\vphantom{2}}_3,x^{\vphantom{2}}_2x^{\vphantom{2}}_4,x_3^2,x_4^2.\)
\end{example}

\begin{table}
\caption{\(G(\alpha)\) values for Example~\ref{example:3}}
{\begin{tabular}{ccccccccccc}%\hline
\(\alpha\) & d\((\alpha)\)&\(G(\alpha)\) &~~~~~~~~~~& \(\alpha\)& d\((\alpha)\) & \(G(\alpha)\) &~~~~~~~~~& \(\alpha\)& d\((\alpha)\) & \(G(\alpha)\)\\ %\hline
\((0,0,0,0)\) &0& - && \((1,1,0,0)\) &2& 28 && \((1,0,1,2)\) & 4 & 33\\
\((0,0,1,0)\) &1& 13 && \(( 0,1,1,1)\) & 3 & 29   &&\((0,1,2,1)\) & 4 & 33  \\
\(( 0,0,0,1)\) &1& 13 &&\((1,0,1,1)\) & 3 & 29 && \(( 1,0,2,1)\) & 4 & 33 \\
\(( 1,0,0,0)\) &1& 19 &&  \(( 0,0,1,2)\) & 3 & 29 &&\(( 0,1,1,2)\) &4& 33 \\
\(( 0,1,0,0)\) &1& 19 && \(( 0,0,2,1)\) & 3 & 29 &&\(( 0,0,2,2)\) & 4 & 33 \\
\(( 0,0,1,1)\) & 2 & 21 && \(( 1,0,2,0)\) & 3 & 31 &&\(( 1,1,2,0)\) & 4 & 34  \\
\(( 1,0,1,0)\) &2& 25 && \(( 1,1,1,0)\) & 3 & 31 &&\(( 1,1,0,2)\) & 4 & 34  \\
\(( 1,0,0,1)\) &2& 25 && \(( 1,1,0,1)\) & 3 & 31 &&\(( 1,1,2,1)\) & 5 & 35  \\
\(( 0,1,1,0)\) &2& 25 && \(( 1,0,0,2)\) & 3 & 31 &&\(( 1,1,1,2)\) & 5 & 35  \\
\(( 0,1,0,1)\) & 2 & 25&& \(( 0,1,2,0)\) & 3 & 31&&\(( 1,0,2,2)\) & 5 & 35  \\
 \(( 0,0,2,0)\) & 2 & 25 && \(( 0,1,0,2)\) & 3 & 31 &&\(( 0,1,2,2)\) & 5 & 35 \\
 \(( 0,0,0,2)\) & 2 & 25  && \(( 1,1,1,1)\) & 4 & 33&&\(( 1,1,2,2)\) & 6 & 36 \\
%\hline
\end{tabular}}
\label{tab1} 
\end{table}

The following corollary to Theorem \ref{inc} gives the smallest number of observations in a design which guarantees inclusion of all monomials of degree one in every leaf of the statistical fan:

\begin{corollary}\label{cormain}
For a design with \(r\) indeterminates each with \(n^{\vphantom{2}}_i=m\), for the monomials of total degree one, equation~\eqref{Geq} becomes \(G({x^{\vphantom{2}}_i})=m^{r-1}+1\).
\end{corollary}

\noindent An application of Corollary \ref{cormain} to factorial designs is well known. For a \(2^k\) factorial, all main effects can be estimated from any subset of \(2^{k-1}+1\) runs.

Careful selection of \(n^{\vphantom{2}}_i\) values is sensible to minimise values of \(G(\alpha)\) for \(\alpha\) of particular interest. In the case of
Example~\ref{example:3}, had \(n^{\vphantom{2}}_i=4\) been chosen for all indeterminates then Theorem~\ref{inc} gives \(G(x^{\vphantom{2}}_1)=65\) and \(G(x_3^2)=129\). Hence, a far
larger design is required to guarantee inclusion of specific monomials in every leaf of the statistical fan.

\section{Identification of {\it HIS} Models via Gr\"obner Bases}\label{sgb}

\subsection{Introduction}\label{gint}
 
\citet{PistoneWynn} use the theory of Gr\"obner bases to find a {\it HIS} model for a given design. The model is unique, contingent on the monomial ordering. The set of {\it HIS} models obtained through the Gr\"obner bases approach is termed the {\it algebraic fan} of the design, and is a subset of the statistical fan. For work on the algebraic fan of a design see~\citet{CPRW}, \citet{PRWb} and~\citet{RKR}. 

In this section, we 
provide a brief overview of key concepts relating to Gr\"obner bases and their use in the construction of {\it HIS} models. Any readers wanting a more comprehensive treatment of Gr\"obner bases are referred to~\citet{CLOS} or~\citet{Froeberg}. Notation relating to Gr\"obner bases is consistent with that of~\citet{CLOS}.  Texts that cover the application of
Gr\"obner bases in the design of experiments include~\citet{PRW1},~\citet{Hibi} and ~\citet{Sullivant}. 

\subsection{Monomial Orderings}\label{sgb1}

The {\G} basis approach to ideals in polynomial rings requires the introduction of an ordering on the space of monomials. 

\begin{definition}[\cite{CLOS}]
\label{order}
A~\emph{monomial  ordering}, \(\preceq\), is a relation on $\mathbb{Z}_{\ge 0}^r$, the set of monomials in $r$ indeterminates, that satisfies the conditions:
\begin{enumerate}
\item[(i)] \(\preceq\) is a total ordering on \(\mathbb{Z}_{\ge 0}^r\);  
\item[(ii)] for all \(\alpha, \beta \in \mathbb{Z}_{\ge 0}^r\) such that \(\alpha \preceq \beta\), we have \(\alpha+\gamma \preceq \beta+\gamma\) for all \(\gamma \in \mathbb{Z}_{\ge 0}^r\); 
\item[(iii)] every non-empty subset of \(\mathbb{Z}_{\ge 0}^r\) has a least element under \(\preceq\).
\end{enumerate}
We write $\alpha \prec \beta$ if $\alpha \preceq \beta$ but $\alpha \neq \beta$. We also refer to the relation $\prec$ as a monomial ordering. 
%We write $\beta \succ \alpha$ if $\alpha \prec \beta$.
\end{definition}
It follows from Definition \ref{order} that any two monomials \(\alpha,\beta\ \in \mathbb{Z}_{\ge 0}^r\)   satisfy exactly one of: \(\alpha\prec \beta\); \(\alpha= \beta\); \(\beta\prec \alpha\).  In the special case $r = 1$, the unique monomial ordering on $\mathbb{Z}_{\ge 0}$ is $0 \prec 1 \prec 2 \prec \dots$. Thus,  for a single indeterminate $x$, the corresponding monomial ordering is $1 \prec x \prec x^2 \prec \dots$. For monomials in more than one indeterminate, there is not a unique monomial ordering. 

A polynomial in $r$ indeterminates is a linear combination of a finite number of monomial terms, which we will list in decreasing order with respect to $\prec$. The first term (i.e. the greatest term with respect to $\prec$) is referred to as the leading term ($\LT$) of the polynomial. For example, the polynomial in one indeterminate \(f(x) = 4x^{11}-2x^6+x^4+x-3\) is written with the terms in decreasing monomial order, and has leading term $\LT (f) = 4 x^{11}$. 

\noindent 

For $r \ge 2$, three  monomial orderings have received particular attention in statistical work. Descriptions follow with the monomial 
orderings expressed in terms of monomials \(\alpha=( \alpha^{\vphantom{2}}_1, \ldots, \alpha^{\vphantom{2}}_r),\beta=( \beta^{\vphantom{2}}_1, \ldots, \beta^{\vphantom{2}}_r ) \in \mathbb{Z}_{\ge 0}^r \). In each case the ordering is illustrated by the monomials with \(r=3\) in the set \(\{(1,2,3), (2,0,4),(2,1,3),(2,1,0),(3,0,0)\}\).

\medskip

\noindent{\bf Lexicographic {\it (lex)}:} \(\alpha\prec_{lex} \beta\) if and only if \(\alpha^{\vphantom{2}}_1 < \beta^{\vphantom{2}}_1\) or there is some \(p \le r\) such that \(\alpha^{\vphantom{2}}_i=\beta^{\vphantom{2}}_i\) for \(i=1,\ldots, p-1\) and \(\alpha^{\vphantom{2}}_p<\beta^{\vphantom{2}}_p\).
\begin{eqnarray*}\label{lord}
(1,2,3)\prec_{lex}(2,0,4)\prec_{lex}(2,1,0)\prec_{lex}(2,1,3)\prec_{lex}(3,0,0).
\end{eqnarray*}
\noindent{\bf Graded Lexicographic {\it (grlex)}:} \(\alpha\prec_{grlex} \beta\) if and only if \({\rm d}(\alpha)<{\rm d}(\beta)\) or \({\rm d}(\alpha)={\rm d}(\beta)\) and \(\alpha\prec_{lex} \beta\). 
\begin{eqnarray*}\label{dlord}
(2,1,0)\prec_{grlex}(3,0,0)\prec_{grlex}(1,2,3)\prec_{grlex}(2,0,4)\prec_{grlex}(2,1,3).
\end{eqnarray*}
\noindent{\bf Graded Reverse Lexicographic {\it (grevlex)}:} \(\alpha\prec_{grevlex} \beta\) if and only if \({\rm d}(\alpha)<{\rm d}(\beta)\) or \({\rm d}(\alpha)={\rm d}(\beta)\) and \(\alpha^{\vphantom{2}}_r> \beta^{\vphantom{2}}_r\) or there is some \(p < r\) such that \(\alpha^{\vphantom{2}}_i=\beta^{\vphantom{2}}_i\) for \(i=p+1,\ldots, r\) and \(\alpha^{\vphantom{2}}_p>\beta^{\vphantom{2}}_p\). 
\begin{eqnarray*}\label{rdlord}
(2,1,0)\prec_{grevlex}(3,0,0)\prec_{grevlex}(2,0,4)\prec_{grevlex}(1,2,3)\prec_{grevlex}(2,1,3).
\end{eqnarray*}

\subsection{Gr\"obner basis criterion}\label{s2s3}

Consider the ring of polynomials $\poly$ in indeterminates $x^{\vphantom{2}}_1, \dots, x^{\vphantom{2}}_r$ over the field $\R$. An~\emph{ideal\/} in $\poly$ is a subset $I \subseteq \poly$ that is a ring in its own right, with the additional property that, for all $f \in \poly$ and all $g \in I$, we have $f \! g, g \! f \in I$. By the Hilbert basis theorem, there exists a finite collection of polynomials $f_1, \dots, f_s \in \poly$ that generate the ideal $I$, in the sense that 
\begin{equation}
\label{lara}
I =  
\left\{ \left. \sum_{i=1}^s h_i f_i \, \right| h_i \in \poly \right\} \eqqcolon \la f_1, \dots, f_s \ra. 
\end{equation}
We define an equivalence relation $\sim$ on $\poly$ by stating that $f_1, f_2 \in \poly$ satisfy $f_1 \sim f_2$ if and only if $f_1 - f_2 \in I$. The space of equivalence classes with respect to $\sim$ carries the structure of a ring, referred to as the~\emph{quotient ring\/}, denoted $\poly / I$. 

\medskip
Given a design $\mathcal{D}$, with $n \times r$ design matrix $X_{\mathcal{D}}$, the rows of $X_{\mathcal{D}}$ define $n$ distinct points $p_1, \dots, p_n$ in $\mathbb{R}^r$ given by 
\[
p_i = \left( x_{i1}, x_{i2}, \dots, x_{ir} \right) \mbox{ for } i = 1, \dots, n. 
\]
We consider the ideal of polynomials that vanish at the points $p_i, i = 1, \dots, n$: 
\[
I_{\mathcal{D}} \coloneqq \left\{ f \in \poly \mid f(p_1) = \dots = f(p_n) = 0 \right\} \subseteq \poly. 
\]
For the points $p_1, \dots, p_n$, the quotient ring $\poly / I_{\mathcal{D}}$  is an $n$-dimensional vector space. 

Given the ideal $I_{\mathcal{D}} \subseteq \poly$ and a monomial ordering, the {\G}-basis construction gives a systematic way of constructing a finite set of polynomials, $\{ g_1, \dots, g_t \}$, that generate the ideal $I_{\mathcal{D}}$~\cite[Chapter~2]{CLOS}. 
Extracting the leading terms of these polynomials with respect to the monomial ordering, $\LT(g_1), \dots, \LT(g_t)$, a basis for the quotient ring is given by the $n$ monomials
\begin{equation}
\left\{ x^{\alpha} \in\poly \mid x^{\alpha} \mbox{ not divisible by $\LT(g_i)$ for $i = 1, \dots, t$} \right\}.
\label{quotientbasis}
\end{equation}
(See, for example,~\cite[Chapter~5, \S3]{CLOS}.) The terms of this basis are then a set of monomials that determine a~{\it HIS} model for the design $\mathcal{D}$. 

Obtaining a Gr\"obner basis for an ideal is generally
computationally intensive and mathematical software packages such
as {\it MAPLE} or {\it MATHEMATICA} should be used. Alternatively,
{\it CoCoA}, a special purpose system for computations in
commutative algebra, is available at~\href{http://cocoa.dima.unige.it}{\texttt{cocoa.dima.unige.it}}. 

To summarise, a design $\mathcal{D}$ defines an ideal $I_{\mathcal{D}} \subseteq \poly$. The strength of the {\G} basis approach is that, given a monomial ordering, it explicitly constructs a basis for $I_{\mathcal{D}}$. This directly leads to a basis for the quotient ring $\poly / I_{\mathcal{D}}$, i.e. a {\it HIS} model for the design $\mathcal{D}$. However, once the monomial ordering is fixed, there is no freedom in this construction while, in general, there may be many possible bases for $I_{\mathcal{D}}$ and $\poly / I_{\mathcal{D}}$. In particular, as the following example shows, there may be bases for $\poly / I_{\mathcal{D}}$ (i.e. \textit{HIS} models for the design $\mathcal{D}$) that~\emph{cannot\/} arise by the {\G} basis construction for~\emph{any\/} choice of monomial ordering. 

\begin{example}
\label{ex:byebyeGroebner}
Consider a design \(\mathcal{D}3\) with $n=7$, $r=4$, having design matrix:
\begin{equation}\label{desmat3}
X_{\mathcal{D}3}=\begin{bmatrix} ~1&~1&~1&~1\\
~1&~1&-1&~1\\
~1&-1&~1&~1\\
~1&-1&-1&-1\\
-1&~1&~1&-1\\
-1&~1&-1&-1\\
-1&-1&~1&-1\\
\end{bmatrix}.
\end{equation}
For both the~\textit{grlex} and \textit{grevlex} orderings, Mathematica gives the following {\G} basis for the ideal $I_{\mathcal{D}3}$: 
\begin{subequations}
\begin{gather}
g_1 = x_1^2 - 1,  \qquad g_2 = x_2^2 - 1, \qquad g_3 = x_3^2 - 1, \qquad g_4 = x_4^2 - 1,
\\
g_5 = x^{\vphantom{2}}_1 \, x^{\vphantom{2}}_2 - x^{\vphantom{2}}_2 \, x^{\vphantom{2}}_4 + x^{\vphantom{2}}_1 - x^{\vphantom{2}}_4, 
\qquad 
g_6 = x^{\vphantom{2}}_1 \, x^{\vphantom{2}}_3 - x^{\vphantom{2}}_3 \, x^{\vphantom{2}}_4 + x^{\vphantom{2}}_1 - x^{\vphantom{2}}_4, 
\\
g_7 = x^{\vphantom{2}}_1 \, x^{\vphantom{2}}_4 + x^{\vphantom{2}}_1 - x^{\vphantom{2}}_4 - 1, 
\qquad 
g_8 = x^{\vphantom{2}}_2 \, x^{\vphantom{2}}_3 - 2 x^{\vphantom{2}}_1 - x^{\vphantom{2}}_2 - x^{\vphantom{2}}_3 + 2 x^{\vphantom{2}}_4 + 1. 
\end{gather}\label{GBlex}\end{subequations}
For both orderings, this gives leading terms 
\begin{gather*}
\LT(g_1) = x_1^2, \quad \LT(g_2) = x_2^2, \quad \LT(g_3) = x_3^2, \quad \LT(g_4) = x_4^2, 
\\
\LT(g_5) = x^{\vphantom{2}}_1 \, x^{\vphantom{2}}_2, \quad \LT(g_6) = x^{\vphantom{2}}_1 \, x^{\vphantom{2}}_3, \qquad \LT(g_7) = x^{\vphantom{2}}_1 \, x^{\vphantom{2}}_4, \quad \LT(g_8) = x^{\vphantom{2}}_2 \, x^{\vphantom{2}}_3. 
\end{gather*}
A basis for the quotient ring is given by the collection of monomials that are not divisible by these leading terms, i.e. the set~\eqref{quotientbasis}. This gives the set of monomials 
\[
\{ 1, \, x^{\vphantom{2}}_1, \, x^{\vphantom{2}}_2, \, x^{\vphantom{2}}_3, \, x^{\vphantom{2}}_4, \, x^{\vphantom{2}}_2 \, x^{\vphantom{2}}_4, \, x^{\vphantom{2}}_3 \, x^{\vphantom{2}}_4 \}
\]
which form a basis for $\poly / I_{\mathcal{D}3}$ and, hence, a {\it HIS} model for the design $\mathcal{D}3$. 

One may check that this is precisely the \textit{HIS} model that would be obtained from the approach of Theorem~\ref{hier} by considering the elements of $T_q$ in the relevant monomial order.  In particular, starting from $1$ and adding terms in increasing \textit{grlex}/\textit{grevlex} order, omitting terms for which the monomial vectors are linearly dependent on those of monomials already included (corresponding to polynomials in the ideal $I_{\mathcal{D}3}$) gives the terms 
\[
1 \prec x^{\vphantom{2}}_4 \prec x^{\vphantom{2}}_3 \prec x^{\vphantom{2}}_2 \prec x^{\vphantom{2}}_1 \prec x^{\vphantom{2}}_3 \, x^{\vphantom{2}}_4 \prec x^{\vphantom{2}}_2 \, x^{\vphantom{2}}_4. 
\]

For the~\textit{lex} ordering, the {\G} basis for $I_{\mathcal{D}3}$ is:
\begin{gather*}
x_2^2 - 1, 
\qquad
x_3^2 - 1, 
\qquad
x_4^2 - 1, 
\\
x^{\vphantom{2}}_2 \, x^{\vphantom{2}}_3 \, x^{\vphantom{2}}_4 + x^{\vphantom{2}}_2 \, x^{\vphantom{2}}_3 - x^{\vphantom{2}}_2 \, x^{\vphantom{2}}_4 - x^{\vphantom{2}}_3 \, x^{\vphantom{2}}_4 - x^{\vphantom{2}}_2 - x^{\vphantom{2}}_3 + x^{\vphantom{2}}_4 + 1, 
\\
x^{\vphantom{2}}_1 - \frac{1}{2} \, x^{\vphantom{2}}_2 \, x^{\vphantom{2}}_3 + \frac{1}{2} \, x^{\vphantom{2}}_2 + \frac{1}{2} \, x^{\vphantom{2}}_3 - x^{\vphantom{2}}_4 - \frac{1}{2}, 
\end{gather*}
with leading terms
\[
x_2^2, \quad 
x_3^2, \quad 
x_4^2, \quad 
x^{\vphantom{2}}_2 \, x^{\vphantom{2}}_3 \, x^{\vphantom{2}}_4, \quad 
x^{\vphantom{2}}_1, 
\]
respectively. A basis for the quotient ring $\poly / I_{\mathcal{D}3}$ and, hence, a {\it HIS} model for the design $\mathcal{D}3$, is given by the set of monomials 
\[
\{ 1, \, x^{\vphantom{2}}_2, \, x^{\vphantom{2}}_3, \, x^{\vphantom{2}}_4, \, x^{\vphantom{2}}_2 \, x^{\vphantom{2}}_3, \, x^{\vphantom{2}}_2 \, x^{\vphantom{2}}_4, \, x^{\vphantom{2}}_3 \, x^{\vphantom{2}}_4 \}. 
\]
This is exactly the \textit{HIS} model that we obtain by an approach similar to that of Theorem~\ref{hier} adding terms in increasing lexicographic order, discarding terms that are linearly dependent on those already included, i.e. 
\[
1 \prec x^{\vphantom{2}}_4 \prec x^{\vphantom{2}}_3 \prec x^{\vphantom{2}}_3 \, x^{\vphantom{2}}_4 \prec x^{\vphantom{2}}_2 \prec x^{\vphantom{2}}_2 \, x^{\vphantom{2}}_4 \prec x^{\vphantom{2}}_2 \, x^{\vphantom{2}}_3. 
\]

In general, for any monomial order, the \textit{HIS} model obtained by the {\G} basis approach can be more easily obtained as follows: Starting with $1$, consider the monomial vectors corresponding to monomials listed in increasing monomial order. A monomial is added in the model if its monomial vector is linearly independent of the monomial vectors already included. Otherwise, it is discarded. In this way, the model is built up term by term. It follows from the {\G} basis construction itself that this construction yields a \textit{HIS} model. 

For the design $\mathcal{D}3$, there are \textit{HIS} models that cannot be obtained from the {\G} basis approach. Such a \textit{HIS} model for $\mathcal{D}3$, that follows from Theorem~\ref{hier}, is the one with monomial set 
\begin{equation}
\left\{ 1, \ x^{\vphantom{2}}_1, \ x^{\vphantom{2}}_2, \ x^{\vphantom{2}}_3, \ x^{\vphantom{2}}_4, \ x^{\vphantom{2}}_1 \, x^{\vphantom{2}}_2, \ x^{\vphantom{2}}_3 \, x^{\vphantom{2}}_4 \right\}. 
\label{monoms}
\end{equation}

\begin{theorem}
The \textit{HIS} model with monomial set~\eqref{monoms} does not arise from the {\G} basis construction for any monomial ordering.
\end{theorem}
\begin{proof}
The proof requires some background from the theory of {\G} bases, which can be found in~\citet[Chapter~2]{CLOS}. 

\medskip

By assumption, there exists a monomial ordering $\prec$,  and corresponding {\G} basis with the property that the monomials in~\eqref{monoms} are exactly the set of monomials that are not divisible by the leading monomials (with respect to $\prec$) of the elements of the {\G} basis. By inspection, the only way that this can be achieved is if the leading terms of the elements of the {\G} basis are
\[
x_1^2, \,  x_2^2, \, x_3^2, \, x_4^2, \, x^{\vphantom{2}}_1 x^{\vphantom{2}}_3, \, x^{\vphantom{2}}_1 x^{\vphantom{2}}_4, \, x^{\vphantom{2}}_2 x^{\vphantom{2}}_3, \, x^{\vphantom{2}}_2 x^{\vphantom{2}}_4. 
\]
Therefore, there must exist a {\G} basis $\{ h_1, \dots, h_8 \}$ for the ideal $I_{\mathcal{D}3}$ with leading terms
\begin{gather*}
\LT(h_1) = x_1^2, \quad \LT(h_2) = x_2^2, \quad \LT(h_3) = x_3^2, \quad \LT(h_4) = x_4^2, 
\\
\LT(h_5) = x^{\vphantom{2}}_1 \, x^{\vphantom{2}}_3, \quad \LT(h_6) = x^{\vphantom{2}}_1 \, x^{\vphantom{2}}_4, \qquad \LT(h_7) = x^{\vphantom{2}}_2 \, x^{\vphantom{2}}_3, \quad \LT(h_8) = x^{\vphantom{2}}_2 \, x^{\vphantom{2}}_4. 
\end{gather*}
The defining property of a {\G} basis is that 
\begin{equation}
\la \LT \left( I_{\mathcal{D}3} \right) \ra = \la \LT(h_1), \dots, \LT(h_8) \ra = \la x_1^2,  x_2^2, x_3^2, x_4^2, x^{\vphantom{2}}_1 x^{\vphantom{2}}_3, x^{\vphantom{2}}_1 x^{\vphantom{2}}_4, x^{\vphantom{2}}_2 x^{\vphantom{2}}_3, x^{\vphantom{2}}_2 x^{\vphantom{2}}_4 \ra, 
\label{GBcondition}
\end{equation}
where $\la \LT(h_1), \dots, \LT(h_8) \ra$ denotes the ideal in $\poly$ generated by the monomials $\LT(h_1), \dots, \LT(h_8)$ (cf.~equation~\eqref{lara}), and $\la \LT \left( I_{\mathcal{D}3} \right) \ra$ denotes the ideal generated by the leading terms of non-zero elements of $I_{\mathcal{D}3}$. 

Consider $g_5 \coloneqq x^{\vphantom{2}}_1 \, x^{\vphantom{2}}_2 - x^{\vphantom{2}}_2 \, x^{\vphantom{2}}_4 + x^{\vphantom{2}}_1 - x^{\vphantom{2}}_4 \in I_{\mathcal{D}3}$, i.e. $g_5$ from~\eqref{GBlex}. Since $1 \prec x^{\alpha}$, for all $\alpha \neq 0$, for any monomial ordering $\prec$, it follows that $x^{\vphantom{2}}_1 \prec x^{\vphantom{2}}_1 x^{\vphantom{2}}_2$ and $x^{\vphantom{2}}_4 \prec x^{\vphantom{2}}_2 \, x^{\vphantom{2}}_4$. Therefore 
\[
\LT(g_5) = \LT \left( x^{\vphantom{2}}_1 \, x^{\vphantom{2}}_2 - x^{\vphantom{2}}_2 \, x^{\vphantom{2}}_4 \right). 
\]
If this were $x^{\vphantom{2}}_1 \, x^{\vphantom{2}}_2$, then $x^{\vphantom{2}}_1 \, x^{\vphantom{2}}_2$ would be an element of the left-hand-side of~\eqref{GBcondition}. However, $x^{\vphantom{2}}_1 \, x^{\vphantom{2}}_2$ is not an element of the right-hand-side of~\eqref{GBcondition}, so this cannot be the case. Therefore, $\LT(g_5) = - x^{\vphantom{2}}_2 \, x^{\vphantom{2}}_4$, so the monomial ordering must satisfy%~\eqref{monord4}. 
\begin{equation}
x^{\vphantom{2}}_1 \, x^{\vphantom{2}}_2 \prec x^{\vphantom{2}}_2 \, x^{\vphantom{2}}_4. 
\label{monord4}
\end{equation}
Similarly, let $g_6 \coloneqq x^{\vphantom{2}}_1 \, x^{\vphantom{2}}_3 - x^{\vphantom{2}}_3 \, x^{\vphantom{2}}_4 + x^{\vphantom{2}}_1 - x^{\vphantom{2}}_4 \in I_{\mathcal{D}3}$, i.e. $g_6$ from~\eqref{GBlex}. We then have 
\[
\LT(g_6) = \LT \left( x^{\vphantom{2}}_1 \, x^{\vphantom{2}}_3 - x^{\vphantom{2}}_3 \, x^{\vphantom{2}}_4 \right). 
\]
Since $- x^{\vphantom{2}}_3 \, x^{\vphantom{2}}_4$ is not an element of the right-hand-side of~\eqref{GBcondition}, we must have $\LT(g_6) = x^{\vphantom{2}}_1 \, x^{\vphantom{2}}_3$, which implies that the monomial ordering satisfies
\begin{equation}
x^{\vphantom{2}}_3 \, x^{\vphantom{2}}_4 \prec x^{\vphantom{2}}_1 \, x^{\vphantom{2}}_3. 
\label{monord1}
\end{equation}

However, by the monomial ordering property, we then have 
\[
x^{\vphantom{2}}_1 \, x^{\vphantom{2}}_2 \, x^{\vphantom{2}}_3 \underset{\eqref{monord4}}{\prec} x^{\vphantom{2}}_2 \, x^{\vphantom{2}}_3 \, x^{\vphantom{2}}_4 \underset{\eqref{monord1}}{\prec} x^{\vphantom{2}}_1 \, x^{\vphantom{2}}_2 \, x^{\vphantom{2}}_3, 
\]
i.e. $x^{\vphantom{2}}_1 \, x^{\vphantom{2}}_2 \, x^{\vphantom{2}}_3 \prec x^{\vphantom{2}}_1 \, x^{\vphantom{2}}_2 \, x^{\vphantom{2}}_3$, which is a contradiction. Therefore, there is no monomial ordering $\prec$ that leads to the desired model. 
\end{proof}

The proof can be understood in terms of the approach of Section~\ref{HIS}, i.e. by constructing the model term by term. From the basis for $I_{\mathcal{D}3}$ given in~\eqref{GBlex}, we deduce that the combinations of monomial vectors $\mathbf{x}^{\vphantom{2}}_1 \, \mathbf{x}^{\vphantom{2}}_2 - \mathbf{x}^{\vphantom{2}}_2 \, \mathbf{x}^{\vphantom{2}}_4 + \mathbf{x}^{\vphantom{2}}_1 - \mathbf{x}^{\vphantom{2}}_4$ and $\mathbf{x}^{\vphantom{2}}_1 \, \mathbf{x}^{\vphantom{2}}_3 - \mathbf{x}^{\vphantom{2}}_3 \, \mathbf{x}^{\vphantom{2}}_4 + \mathbf{x}^{\vphantom{2}}_1 - \mathbf{x}^{\vphantom{2}}_4$ are identically zero for the design $\mathcal{D}3$. (This can be checked directly from the design matrix.) The first statement implies that the collection of column vectors $\left\{ \mathbf{x}^{\vphantom{2}}_1, \mathbf{x}^{\vphantom{2}}_4, \mathbf{x}^{\vphantom{2}}_1 \, \mathbf{x}^{\vphantom{2}}_2, \mathbf{x}^{\vphantom{2}}_2 \, \mathbf{x}^{\vphantom{2}}_4 \right\}$ is linearly dependent. Therefore, if we wish to include the monomials $x_1, x_4$ and $x_1 \, x_2$ in our model, and we are adding terms in increasing monomial order, then $x_1, x_4$ and $x_1 \, x_2$ must preceed $x_2 \, x_4$. The statement that $x_1 \, x_2$ preceeds $x_2 \, x_4$ is precisely the condition~\eqref{monord4} on $\prec$. Similarly, the vectors $\left\{  \mathbf{x}^{\vphantom{2}}_1, \mathbf{x}^{\vphantom{2}}_4, \mathbf{x}^{\vphantom{2}}_1 \, \mathbf{x}^{\vphantom{2}}_3, \mathbf{x}^{\vphantom{2}}_3 \, \mathbf{x}^{\vphantom{2}}_4 \right\}$ are linearly dependent. Since we require $x_1$, $x_4$ and $x_3 \, x_4$ in our model, the monomial ordering $\prec$ must satisfy the condition~\eqref{monord1}. This leads to the same contradiction as above. 

To summarise, the constraints imposed by any monomial ordering prohibit the inclusion of both $x_1 \, x_2$ and $x_3 \, x_4$ in a model for the design $\mathcal{D}3$. As such, \textit{HIS} models for $\mathcal{D}3$ with monomials as in~\eqref{monoms} can easily be generated by our approach, but~\emph{not\/} by the {\G} basis approach. 
\end{example}

\begin{example}[continues=thm:key2]
\label{example:4cont}
In~\citet{MAW}, the authors state that the algebraic fan for  design \(\mathcal{D}2\) 
comprises 54 {\it HIS} models, arranged in three classes. A representative from each class is obtained by the {\G} basis approach from the standard monomial orderings as follows: 

\medskip

\noindent{\bf Class~1:}  {\it grlex} and {\it grevlex} both yield the model with terms: $1,x^{\vphantom{2}}_1,x^{\vphantom{2}}_2,x^{\vphantom{2}}_3,x^{\vphantom{2}}_4, x_3^2, x^{\vphantom{2}}_3x^{\vphantom{2}}_4, x_4^2, x^{\vphantom{2}}_2x^{\vphantom{2}}_4$. 

\medskip
\noindent{\bf Class 2:} {\it grlex} and {\it grevlex} on the indeterminate subset \(x^{\vphantom{2}}_2,x^{\vphantom{2}}_3,x^{\vphantom{2}}_4\) both yield the model with terms: 
$1, x^{\vphantom{2}}_2,x^{\vphantom{2}}_3,x^{\vphantom{2}}_4,x^{\vphantom{2}}_2x^{\vphantom{2}}_4,x^{\vphantom{2}}_3x^{\vphantom{2}}_4,x_3^2,x_4^2,x^{\vphantom{2}}_3x_4^2$.

\medskip
\noindent{\bf Class 3:} {\it lex} yields the model with terms: $1, x^{\vphantom{2}}_3,x^{\vphantom{2}}_4,x_4^2,x^{\vphantom{2}}_3x^{\vphantom{2}}_4,x^{\vphantom{2}}_3x_4^2,x_3^2,x_3^2x^{\vphantom{2}}_4,x_3^2x_4^2$. 

\smallskip

By considering the $4!$ permutations of the indeterminates $x_1, \dots, x_4$, and the corresponding monomial orderings, \citet{MAW} obtain a further $23$ models in each of Classes 1 and 2 and a further $5$ models in Class 3. These have the terms listed above with the subscripts permuted. All $54$ models in the algebraic fan can be obtained via 
the approach of Theorem \ref{hier} and Corollary \ref{newc}. To illustrate this, 
the Class 1 model given above is obtained by considering vectors in \(T_1\) in any order and then vectors in   \( T_2\) in the {\it 
grevlex} order of the corresponding monomials, i.e. in the order
\begin{equation*}
\mathbf{x}_4^2,\mathbf{x}^{\vphantom{2}}_3 \mathbf{x}^{\vphantom{2}}_4, \mathbf{x}^{\vphantom{2}}_2 \mathbf{x}^{\vphantom{2}}_4, \mathbf{x}^{\vphantom{2}}_1\mathbf{x}^{\vphantom{2}}_4, \mathbf{x}_3^2, \mathbf{x}^{\vphantom{2}}_2 \mathbf{x}^{\vphantom{2}}_3, \mathbf{x}^{\vphantom{2}}_1 \mathbf{x}^{\vphantom{2}}_3, \mathbf{x}_2^2, \mathbf{x}^{\vphantom{2}}_1 \mathbf{x}^{\vphantom{2}}_2, \mathbf{x}_1^2. 
\end{equation*}
By comparison, $251$ {\it HIS} models were found by the computationally simpler methods of Theorem \ref{hier} and Corollary \ref{newc}.
\end{example}

\subsection{Comparison between methods}

The model building approach of  \S\ref{HIS} and \S\ref{s3} and the {\G}  basis approach of \S\ref{sgb}, when used with a graded 
monomial ordering, share the common feature of prioritising monomials with smaller total degree for model inclusion. This seems intuitively desirable. 

The approach of \S\ref{HIS} and \S\ref{s3}, specified through Theorem \ref{hier} and Corollary \ref{newc}, has three key advantages. First, it allows for user input, enabling prior knowledge of the process associated with the experiment to guide the order in which terms are considered for model inclusion at each level. Second, the methods require only elementary linear algebra. Third, as demonstrated by Example \ref{example:4cont}, the models identified using the {\G} basis approach with a graded monomial ordering are a subset of those found using Theorem \ref{hier}. Thus, the methods of \S\ref{HIS} and \S\ref{s3} are both flexible and accessible.  Further work could explore whether this approach can be adapted to include terms hierarchically, but not necessarily in a monotonic manner with respect to monomial degree.

The appeal of the {\G} basis approach is that, given a monomial ordering, a  unique {\it HIS} model is obtained, 
requiring no practitioner input. However, a small amount of practitioner input is possible. For instance a 
reordering of the indeterminates can prioritise the inclusion of desired terms. The 
computationally intensive nature of the construction of a {\G} basis does mean that the method is reliant on the use of packages that can conduct computations in commutative algebra. More fundamentally, the order in which monomials of the same degree are considered for model inclusion is constrained by the monomial ordering used. This is a direct consequence of the {\G} basis approach. It is challenging to imagine a situation in which this constraint would be relevant to a real-world problem. 

%\newpage

\providecommand{\bysame}{\leavevmode\hbox to3em{\hrulefill}\thinspace}

\end{document}